\documentclass[twoside,12pt,a4paper]{amsart}
\usepackage{amssymb}

\usepackage{amsmath}
\usepackage{hyperref}
\newtheorem{theorem}{Theorem}[section]

\newtheorem{corollary}[theorem]{Corollary}
\newtheorem{definition}[theorem]{Definition}

\newtheorem{lemma}[theorem]{Lemma}

\newtheorem{proposition}[theorem]{Proposition}

\def\uh{\upharpoonright}

\newcommand{\CK}{\omega_1^{\mathrm{CK}}}
\newcommand{\KO}{\mathcal{O}}

\newcommand{\QI}{\Pi^1_1}
\newcommand{\DI}{\Delta^1_1}

\begin{document}
\title{Higher Kurtz randomness}
\author{Bj{\o}rn Kjos-Hanssen}
\address{Department of Mathematics\\
University of Hawaii at Manoa\\ 2565 McCarthy Mall \\ Honolulu, HI
96822 \\ USA} \email{bjoern@math.hawaii.edu}

\author{Andr\'e Nies}
\address{Department of Computer Science\\
  University of Auckland\\
  Private Bag 92019\\ Auckland \\ New Zealand } \email{andrenies@gmail.com}

\author{Frank Stephan}
\address{Departments of Mathematics and Computer Science\\
National University of Singapore, 2 Science Drive 2, Singapore 117543\\
Republic of Singapore.}
\email{fstephan@comp.nus.edu.sg}

\author{Liang Yu}
\address{Institute of Mathematical Science\\
Nanjing University, Nanjing, JiangSu Province, 210093\\
P.R. of China} \email{yuliang.nju@gmail.com}

\begin{abstract}
A real $x$ is $\DI$-Kurtz random  ($\QI$-Kurtz random) if it is in no closed null $\DI$ set ($\QI$ set).  We show that there is a cone of $\QI$-Kurtz random hyperdegrees.  We characterize lowness for $\DI$-Kurtz randomness as being $\DI$-dominated and $\DI$-semi-traceable.
\end{abstract}

\thanks{Kjos-Hanssen's research was partially supported by NSF (U.S.A.) grants DMS-0652669 and DMS-0901020. Nies is partially supported by the Marsden Fund of New Zealand, grant No.\ 08-UOA-184. Stephan is supported in part by NUS grants number R146-000-114-112 and R252-000-308-112. Yu is supported by NSF of China No.\ 10701041 and Research Fund for the Doctoral Program of Higher Education, No.\ 20070284043.}

\maketitle

\section{Introduction}

\noindent
Traditionally one uses tools from recursion theory to obtain mathematical notions corresponding to our intuitive idea of  randomness
 for reals. However, already Martin-L{\"o}f \cite{ML70} suggested to
use tools from higher recursion (or equivalently, effective
descriptive set theory) when he introduced the notion of
$\DI$-randomness. This approach was pursued to greater depths  by
Hjorth and Nies \cite{HjNi} and Chong, Nies and Yu \cite{CNY07}.
 Hjorth and Nies investigated a higher analog of the usual
Martin-L{\"o}f randomness, and a new notion with no direct analog in (lower)
recursion theory: a real is $\QI$-random if it avoids each null $\QI$
set. Chong, Nies and Yu \cite{CNY07} studied $\DI$-randomness in more
detail, viewing it as a higher analog of both Schnorr and recursive
randomness. By  now a classical result is the characterization of
lowness for Schnorr randomness by recursive traceability (see, for
instance, Nies' textbook \cite{Nies:book}). Chong, Nies and Yu \cite{CNY07}  proved a
higher analog of this result,  characterizing lowness for $\DI$
randomness by $\DI$ traceability.

  Our goal is to carry out similar investigations for higher analogs
of Kurtz randomness \cite{Downey.Griffiths.ea:04}. A real $x$ is Kurtz
random if avoids each $\Pi^0_1$ null class. This is quite a weak notion of randomness: each weakly 1-generic set is Kurtz random, so for instance the law of large numbers can fail badly.

It is essential for Kurtz randomness that the tests are {\it closed} null sets.     For higher analogs of Kurtz randomness one can require that these tests are closed and belong to a more permissive class such as $\DI$,  $\QI$, or $\Sigma^1_1$.

   Restrictions on the computational complexity of a real have been used successfully to analyze randomness notions. For instance, a Martin-L{\"o}f random real  is weakly 2-random iff it forms a minimal pair  with $\emptyset'$ (see \cite{Nies:book}). We prove a result of that kind in the present setting. Chong, Nies, and Yu \cite{CNY07} studied a property restricting the complexity of a real: being  $\DI$-dominated. This is the higher analog of being recursively dominated (or of hyperimmune-free degree).  We show that
    a $\DI$-Kurtz random $\DI$ dominated set is already $\QI$-random.
Thus $\DI$-Kurtz randomness is equivalent to a proper randomness notion on a conull set. We also study the distribution of higher Kurtz random reals in the hyperdegrees.  For instance,  there is a cone of $\QI$-Kurtz random   hyperdegrees. However, its base is very complex, having the largest hyperdegree among all $\Sigma^1_2$ reals.

    Thereafter we turn to lowness for higher Kurtz randomness. Recursive traceability of a real  $x$  is easily seen  to be equivalent to the condition that for each function  $f \le_T x$ there is a recursive function $\hat f$ that agrees with $f$ on at least one
input in each interval of the form $[2^n, 2^{n+1}-1)$ (see \cite[8.2.21]{Nies:book}). Following Kjos-Hanssen, Merkle, and Stephan \cite{KMSarXiv} one says that $x$ is recursively semi-traceable (or infinitely often traceable) if for each $f \le_T x$ there is a recursive function $\hat f$ that agrees with $f$ on infinitely many inputs. It is straightforward to define the higher analog of this
notion, $\DI$-semi-traceability. Our main result is that lowness for $\DI$-Kurtz randomness is equivalent to being $\DI$-dominated and $\DI$-semi-traceable. We also show using forcing that being $\DI$-dominated and $\DI$-semi-traceable is strictly weaker than being $\DI$-traceable. Thus, lowness for $\DI$ Kurtz randomness is strictly
weaker than lowness for $\DI$-randomness.

   \bigskip

\section{Preliminaries}

\noindent
We assume that the reader is familiar with elements of higher
recursion theory, as  presented, for instance,  in Sacks
\cite{Sacks90}. See \cite[Ch. 9]{Nies:book} for a summary.

 A real is an element in $2^{\omega}$. Sometimes we
write $n\in x$ to mean $x(n)=1$. Fix a standard $\Pi^0_2$ set
$H\subseteq \omega\times 2^{\omega} \times 2^{\omega}$ so that for
all $x$ and $n \in \KO$, there is a unique real $y$ satisfying
$H(n,x,y)$. Moreover, if $\omega_1^x=\CK$, then each real $z\leq_h
x$ is Turing reducible to some $y$ so that $H(n,x,y)$ holds for some
$n \in \KO$.  Roughly speaking, $y$ is the $|n|$-th Turing jump of
$x$. These $y$'s are called $H^x$ sets and denoted by $H^x_n$. For
each $n\in \KO$, let $\KO_n=\{m\in \KO\mid |m|<|n|\}$. $\KO_n$ is a
$\Delta^1_1$ set.

We use the Cantor pairing function, the bijection
$p:\omega^2\to\omega$ given by $p(n,s)=\frac{(n+s)^2+3n+s}{2}$, and
write $\langle n,s\rangle=p(n,s)$. For a finite string $\sigma$,
$[\sigma]=\{x \succ \sigma\mid x \in 2^{\omega}\}$.
For an open set $U$, there is a presentation $\hat{U}\subseteq 2^{<\omega}$ so that $\sigma
\in \hat{U}$ if and only if $[\sigma]\subseteq U$. We sometimes
identify $U$ with $\hat{U}$.
For a recursive functional $\Phi$, we use $\Phi^{\sigma}[s]$ to denote
the computation state of $\Phi^{\sigma}$ at stage $s$. For a tree $T$,
we use $[T]$ to denote the set of infinite paths in $T$. Some times we
identify a finite string $\sigma\in \omega^{<\omega}$ with a natural
number without confusion.

The following results will be used in later sections.
\bigskip

\begin{theorem}[Gandy]\label{theorem lowbase} If $A\subseteq 2^\omega$
is a nonempty $\Sigma^1_1$
set, then there is a real $x\in A$ so that $\mathcal{O}^x \leq_h
\mathcal{O}$.
\end{theorem}

\begin{theorem}[Spector \cite{Spector59} and Gandy
\cite{Gandy60}]\label{theorem spector and gandy} $A\subset 2^\omega$
is $\Pi^1_1$ if and only if there is an
arithmetical predicate $P(x,y)$ such that
$$y \in A \leftrightarrow
\exists x\leq_h yP(x,y).
$$
\end{theorem}

\begin{theorem}[Sacks\cite{Sack69}]\label{theorem cone null}
If $x$ is non-hyperarithmetical, then $\mu(\{y|y\geq_h x\})=0$.
\end{theorem}

\begin{theorem}[Sacks \cite{Sacks90}]\label{theorem x h>o iff x computes
longer}
The set $\{x|x \geq_h \KO\}$ is $\Pi^1_1$. Moreover, $x \geq_h \KO$
if and only if $\omega^x_1>\CK$.
\end{theorem}

\noindent
A consequence of the last two theorems above is that the set
$\{x\mid\omega_1^x>\CK\}$ is a $\Pi^1_1$ null set.

Given a class $\mathbf{\Gamma}$, an element $x\in \omega^{\omega}$
is called a {\em $\mathbf{\Gamma}$-singleton} if $\{x\}$ is a
$\mathbf{\Gamma}$ set. Note that if $x\in \omega^{\omega}$ is a
$\Pi^1_1$-singleton, then too is $x_0=\{\langle n,m\rangle \mid
x(n)=m\}\equiv_T x$. Hence we do not distinguish
$\Pi^1_1$-singletons between Baire space and Cantor space.

A subset of $2^{\omega}$ is $\mathbf{\Pi}^0_0$ if it is clopen. We
can define $\mathbf{\Pi}^0_{\gamma}$ sets by a transfinite induction
  for all countable $\gamma$. Every such set can
be coded by a real  (for more details see
\cite{Sacks90}). Given a class $\mathbf{\Gamma}$ (for example,
$\mathbf{\Gamma}=\Delta^1_1$) of subsets of $2^{\omega}$, a set $A$ is
$\mathbf{\Pi}^0_{\gamma}(\mathbf{\Gamma})$ if $A$ is
$\mathbf{\Pi}^0_{\gamma}$ and can be coded by a real in~$\mathbf{\Gamma}$.

In the case $\gamma=1$, every hyperarithmetic closed subset of reals
is $\mathbf{\Pi^0_1}(\Delta^1_1)$. We also have the following result
with an easy proof.
\begin{proposition}\label{proposition: separate sigma11}
If $A\subseteq 2^{\omega}$ is $\Sigma^1_1$ and $\mathbf{\Pi}^0_1$,
then $A $ is $\mathbf{\Pi}^0_1(\Sigma^1_1)$.
\end{proposition}
\begin{proof}
Let $z=\{\sigma\mid \exists x(x \in A \wedge x \succ \sigma)\}$.
Then $x \in A$ if and only if $\forall n(x \uh n \in z)$. So $A$ is
$\Pi^0_1(z)$. Obviously $z$ is $\Sigma^1_1$.
\end{proof}

\noindent
Note that Proposition \ref{proposition: separate sigma11} fails if
we replace $\Sigma^1_1$ with $\Pi^1_1$ since $\KO^{\KO}$ is a
$\Pi^1_1$ singleton of hyperdegree greater than $\KO$.

\bigskip
\noindent
The ramified analytical hierarchy was introduced by Kleene, and
applied by Fefferman \cite{Fefer65} and Cohen \cite{Cohen66} to
study forcing,  a tool that turns out to be powerful  in the
investigation of higher randomness theory. We recall some basic
facts from  Sacks \cite{Sacks90} whose notations we mostly
follow:

\bigskip
\noindent
The ramified analytic hierarchy language $\mathfrak{L}(\CK,
\dot{x})$ contains the following symbols:
\begin{enumerate}
\item Number variables: $j,k,m,n,\ldots$;
\item Numerals: $0,1,2,\ldots$;
\item Constant: $\dot{x}$;
\item Ranked set variables: $x^{\alpha},y^{\alpha},\ldots$ where
$\alpha<\CK$;
\item Unranked set variables: $x,y,ldots$;
\item Others symbols include: +, $\cdot$ (times), $'$ (successor) and $\in$.
\end{enumerate}
Formulas are built in the usual way. A formula $\varphi$ is {\em
ranked} if all of its set variables are ranked. Due to its
complexity, the language is not  codable in   a recursive set but
rather in the countable admissible set $L_{\CK}$.

To code the language in a uniform way, we fix a $\Pi^1_1$ path
$\KO_1$ through $\KO$ (by \cite{FeSpe62} such a path exists). Then a
ranked set variable $x^{\alpha}$ is coded by the number $(2,n)$
where $n \in \KO_1$ and $|n|=\alpha$. Other   symbols and formulas
are coded recursively. With such a coding,  the set of G\"odel
number of  formulas is $\Pi^1_1$. Moreover, the set of G\"odel
numbers of ranked formulas of rank less than $\alpha$ is r.e.
uniformly in the unique notation for $\alpha$ in $\KO_1$. Hence
there is a recursive function $f$ so that $W_{f(n)}$ is the set of
G\"odel numbers of the ranked formula of rank less than $|n|$ when
$n \in \KO_1$ ($\{W_e\}_e$ is, as usual,  an effective enumeration
of r.e. sets).

One now defines a structure $\mathfrak{A}(\CK, x)$, where $x$ is a
real, analogous to the way G\"odel's $L$ is defined, by induction on
the recursive ordinals. Only at successor stages are new sets
defined in the structure.
 The  reals constructed at a successor stage are arithmetically
definable from the reals constructed at earlier stages. The details
may be found in \cite{Sacks90}. We define $\mathfrak{A}(\CK,
x)\models \varphi$ for a formula $\varphi$ of $\mathfrak{L}(\CK,
\dot{x})$ by allowing the unranked set variables to range over
$\mathfrak{A}(\CK, x)$, while the symbol  $x^{\alpha}$ will be
interpreted as  the reals built before stage $\alpha$. In fact, the
domain of $\mathfrak{A}(\CK, x)$ is the set $\{y\mid y \leq_h x\}$ if
and only if $\omega_1^x =\CK$ (see \cite{Sacks90}).

A sentence $\varphi$ of $\mathfrak{L}(\CK, \dot{x})$ is said to be
$\Sigma^1_1$ if it is ranked, or of the form $\exists
x_1,\ldots,\exists x_n\psi$ for some formula $\psi$ with no unranked
set variables bounded by a quantifier.

The following result   is a model-theoretic version of
the Gandy-Spector Theorem.

\begin{theorem}[Sacks \cite{Sacks90}]\label{theorem ramified sg thm}
The set $\{(n_{\varphi},x)\mid\varphi \in \Sigma^1_1\wedge
\mathfrak{A}(\CK, x)\models  \varphi \}$ is $\Pi^1_1$,  where
$n_{\varphi}$ is the G\"odel number of $\varphi$. Moreover, for each
$\Pi^1_1$ set $A \subseteq 2^{\omega}$, there is a formula $\varphi
\in \Sigma^1_1$ so that
\begin{enumerate} \item $\mathfrak{A}(\CK, x)\models
\varphi \implies x \in A$;
\item if $\omega^x_1=\CK$, then $ \mathfrak{A}(\CK, x)\models \varphi \Longleftrightarrow
x \in A .$\end{enumerate}
\end{theorem}

\noindent
Note that if $\varphi$ is ranked, then both the sets
$\{x\mid\mathfrak{A}(\CK, x)\models \varphi\}$ (the G\"odel number of
$\varphi$ is omitted) and $\{x\mid\mathfrak{A}(\CK, x)\models
\neg\varphi\}$  are $\Pi^1_1$. So both sets  are  $\Delta^1_1$. Moreover, if $A
\subseteq 2^{\omega}$ is $\Delta^1_1$, then there is a ranked
formula $\varphi$ so that $x \in A \Leftrightarrow \mathfrak{A}(\CK,
x)\models \varphi$ (see Sacks \cite{Sacks90}).

\begin{theorem}[Sacks \cite{Sack69}]\label{theorem pi11 measure is
pi} The set
\begin{center} $\{(n_{\varphi},p)\mid\mu(\{x\mid\mathfrak{A}(\CK, x)\models
\varphi\})>p \wedge \varphi \in \Sigma^1_1 \wedge p \mbox{ is a
rational number} \}$ \end{center}
 is $\Pi^1_1$ where $n_{\varphi}$ is the G\"odel
number of $\varphi$.
\end{theorem}

\begin{theorem}[Sacks \cite{Sack69}]\label{theorem approximating delta11 set}
There is a recursive function $f:\omega \times \omega \to
\omega$ so that for all $n $ which is  G\"odel number of a ranked
formula:
\begin{enumerate}
\item $f(n,p)$ is  G\"odel number of a
ranked formula;
\item the set $\{x\mid\mathfrak{A}(\CK, x)\models
\varphi_{f(n,p)}\}\supseteq \{x\mid\mathfrak{A}(\CK, x)\models
\varphi_n\} $ is open; and
\item $\mu(\{x\mid\mathfrak{A}(\CK, x)\models
\varphi_{f(n,p)}\}- \{x\mid\mathfrak{A}(\CK, x)\models \varphi_n\}
)<\frac{1}{p}$.
\end{enumerate}
\end{theorem}

\begin{theorem}[Sacks \cite{Sack69} and Tanaka
\cite{Tanaka68}]\label{theorem pi11 measure base thm} If $A$ is a
$\Pi^1_1$ set of positive measure, then $A$ contains a hyperarithmetical real.
\end{theorem}

We also remind the reader of the higher analog of ML-randomness first studied by \cite{HjNi}.
\begin{definition}
A \emph{$\Pi^1_1$-ML-test} is a sequence $(G_m)_{m\in\omega}$ of open sets such that for each $m$, we have $\mu(G_m)\le 2^{-m}$, and the relation $\{\langle m,\sigma\rangle\mid [\sigma]\subseteq G_m\}$ is $\Pi^1_1$. A real $x$ is \emph{$\Pi^1_1$-ML-random} if $x\not\in\cap_m G_m$ for each $\Pi^1_1$-ML-test $(G_m)_{m\in\omega}$.
\end{definition}

\bigskip
\section{Higher Kurtz random reals and their distribution}

\begin{definition}
Suppose we are given a point class $\mathbf{\Gamma}$ (i.e. a class of sets of
reals). A real $x$ is {\em $\mathbf{\Gamma}$-Kurtz random} if $ x \not\in A$ for
every closed null set $A \in \mathbf{\Gamma}$. Further, $x$ is
said to be {\em Kurtz random} ($y$-Kurtz random) if
$\mathbf{\Gamma}=\Pi^0_1$ $(\mathbf{\Gamma}=\Pi^0_1(y))$.
\end{definition}

\noindent We focus on $\Delta^1_1$, $\Sigma^1_1$ and $\Pi^1_1$-Kurtz
randomness. By the proof of Proposition \ref{proposition: separate
sigma11}, it is not difficult to see that a real $x$ is
$\Delta^1_1$-Kurtz random  if and only if $x$ does not belong to any
$\mathbf{\Pi}^0_1(\Delta^1_1)$  null set.

\begin{theorem}\label{theorem seperating randomness}
$\Pi^1_1$-Kurtz randomness $\subset$ $\Sigma^1_1$-Kurtz randomness
$=$ $\Delta^1_1$-Kurtz-ran\-dom\-ness.
\end{theorem}
\begin{proof}

\noindent It is obvious that $\Pi^1_1$-Kurtz randomness $\subseteq
\Delta^1_1$-Kurtz randomness and $\Sigma^1_1$-Kurtz randomness
$\subseteq\Delta^1_1$-Kurtz randomness. It suffices to prove that
$\Sigma^1_1$-Kurtz randomness $=$ $\Delta^1_1$-Kurtz-ran\-dom\-ness
and $\Pi^1_1$-Kurtz randomness $\subset \Delta^1_1$-Kurtz
randomness.

Note that every $\Pi^1_1$-ML-random is $\Delta^1_1$-Kurtz random and
there is a $\Pi^1_1$-ML-random real $x \equiv_h \KO$  (see
\cite{HjNi} and \cite{CNY07}). But $\{x\}$ is a $\Pi^1_1$ closed
set. So $x$ is not $\Pi^1_1$-Kurtz random. Hence
$\Pi^1_1$-Kurtz randomness $\subset \Delta^1_1$-Kurtz randomness.

Suppose we are given a $\Pi^1_1$ open set $A$ of measure 1. Define
$$
x=\{\sigma\in
2^{<\omega}\mid \forall y(y\succ \sigma \implies y \in A)\}.
$$
Then $x$ is a $\Pi^1_1$ real coding $A$ (i.e. $y\in A$ if and only
if there is a $\sigma \in x$ for which $y \succ \sigma$, or $y\in
[\sigma]$). So there is a recursive function $f:2^{<\omega} \to
\omega$ so that $\sigma \in x$ if and only if $f(\sigma)\in \KO$.
Define a $\Pi^1_1$ relation $R\subseteq \omega\times \omega$ so that
$(k,n)\in R$ if and only if $n \in \KO$ and $\mu(\bigcup
\{[\sigma]\mid \exists m\in \KO_n(f(\sigma)=m)\})>1-\frac{1}{k}$.
Obviously $R$ is a $\Pi^1_1$ relation which can be uniformized by a
$\Pi^1_1$ function $f^*$ (see \cite{Mosch1980}). Since $\mu(A)=1$,
$f^*$ is a total function. So the range of $f^*$ is bounded by a
notation $n \in \KO$. Define $B=\{y \mid \exists \sigma (y \succ
\sigma \wedge f(\sigma) \in \KO_n)\}$. Then $B\subseteq A$ is a
$\Delta^1_1$ open set with measure 1. So every $\Pi^1_1$ open conull
set has a $\Delta^1_1$ open conull subset. Hence $\Sigma^1_1$-Kurtz
randomness equals $\Delta^1_1$-Kurtz randomness.
\end{proof}

\noindent It should be pointed out that, by the proof of Theorem \ref{theorem seperating randomness}, not every $\Pi^1_1$-ML-random real is $\Pi^1_1$-Kurtz random.

\bigskip

\noindent
The following result clarifies the relationship between $\Delta^1_1$- and
$\Pi^1_1$-Kurtz randomness.

\begin{proposition}\label{proposition: delta11 vs pi11 at ck}
If $\omega_1^x=\CK$, then $x$ is $\Pi^1_1$-Kurtz random if and only if
$x$ is $\Delta^1_1$-Kurtz random.
\end{proposition}

\begin{proof}
Suppose that $\omega_1^x=\CK$ and $x$ is $\Delta^1_1$-Kurtz random.
If $A$ is a $\Pi^1_1$ closed null set so that $x \in A$, then by
Theorem \ref{theorem ramified sg thm}, there is a  formula
$\varphi(z,y)$ whose only unranked set variables are $z$ and $y$  so
that the formula $\exists z \varphi(z,y)$ defines $A$.  Since
$\omega_1^x=\CK$, $x \in B=\{y\mid \mathfrak{A}(\CK, y)\models \exists
z^{\alpha}\varphi(z^{\alpha},y) \}\subseteq A$ for some recursive
ordinal $\alpha$. Define $T=\{\sigma\in 2^{<\omega}\mid \exists y \in
B(y \succ \sigma)\}$.  Obviously $B\subseteq [T]$. Since $B$ is
$\Delta^1_1$, $[T]$ is $\Sigma^1_1$. Since $A$ is closed, $B\subseteq A$, and $[T]$ is the closure of $B$, we have $[T]\subseteq A$. Hence since $A$ is null, so is $[T]$. By the proof of Theorem \ref{theorem seperating randomness}, there is a $\Delta^1_1$ closed null set $C \supseteq [T]$. Hence $x \in C$, a contradiction.
\end{proof}

\noindent
From the proof of Theorem \ref{theorem seperating randomness}, one
sees that every hyperdegree above $\KO$ contains a $\Delta^1_1$-Kurtz random real. But this fails for $\Pi^1_1$-Kurtz randomness. We say that a hyperdegree
$\mathbf{d}$ is a {\em base for a cone of   $\mathbf{\Gamma}$-Kurtz randoms}
if for every hyperarithmetic degree $\mathbf{h}\geq\mathbf{d}$,
$\mathbf{h}$ contains a $\mathbf{\Gamma}$-Kurtz random real.

The hyperdegree of $\KO$ is a base for a cone of  $\Delta^1_1$-Kurtz randoms as
proved in Theorem \ref{theorem seperating randomness}. In   Corollary \ref{corollary: nonbase for delta11 kurtz random} we will show that not every nonzero hyperdegree is a base of a cone of $\Delta^1_1$-Kurtz randoms.

Is there a base for a cone of $\Pi^1_1$-Kurtz randoms? If
  such a base $\mathbf{b}$ exists, then $\mathbf{b}$ is not
hyperarithmetically reducible to any $\Pi^1_1$ singleton. Intuitively, this means that such
bases must be   complex.

To obtain such a base we need a lemma.

\begin{lemma}\label{lemma Stephan lemma}
For any reals $x$ and $z\geq_T x'$, there is an $x$-Kurtz random
real $y\equiv_T z$.
\end{lemma}

\begin{proof}
  Fix an enumeration of the
$x$-r.e.\ open sets $\{U^x_n\}_{n\in \omega}$.

We inductively define an increasing sequence of binary strings $\{\sigma_s\}_{s<\omega}$.

\bigskip
\noindent
Stage 0. Let $\sigma_0$ be the  empty string.

\bigskip
\noindent
Stage $s+1$. Let $l_0=0$, $l_1=|\sigma_s|$, and $l_{n+1}=2^{l_n}$
for all $n>1$. For every $n>1$, let $$A_n=\{\sigma \in 2^{l_n-1}\mid
\exists m<n\forall i\forall j(l_m\leq i, j<l_{m+1}\implies
\sigma(i)=\sigma(j)) \}.$$ Then $$|A_n|\leq 2\cdot 2^{l_{n-1}}.$$
In other words,
$$\mu(\bigcup \{[\sigma]\mid \sigma \succeq \sigma_s\wedge \sigma
\not\in A_n\})\geq 2^{-l_1}\cdot (1-2^{l_n+1-l_{n+1}}).$$
Case(1): There is  some $m>l_1+1$ so that $|\{\sigma\succeq
\sigma_s\mid \sigma\in 2^m \wedge [\sigma]\subseteq
U_s^x\}|>2^{m-l_1-1}$. Let $n= m+1$. Then $l_{n+1}-1-l_n>2$ and
$l_n>m$. So there must be some $\sigma \in 2^{l_n-1}-A_n$ so that
there is a $\tau \preceq \sigma $ for which $[\tau]\subseteq U^x_s$
and $\tau \in 2^m$.

Let $\sigma_{s+1}=\sigma^{\smallfrown}(z(s))^{l_{n}-1}$.

\bigskip
\noindent
Case(2): Otherwise. Let
$\sigma_{s+1}=\sigma_s^{\smallfrown}(z(s))^{l_1-1}$.

This finishes the construction at stage $s+1$.

\bigskip
\noindent
Let $y=\bigcup_s \sigma_s$.

Obviously the construction is recursive in $z$. So $y\leq_T z$.
Moreover, if $U_n^x$ is of measure 1, then Case (1) happens at the
stage $n+1$. So $y$ is $x$-Kurtz random.

Let $l_0=0, l_{n+1}=2^{l_n}$ for all $n\in \omega$. To compute
$z(n)$ from $y$, we $y$-recursively find the $n$-th $l_m$ for which
for all $i,j$ with $l_m\leq i<j<l_{m+1}$, $y(i)=y(j)$. Then
$z(n)=y(l_m)$.
\end{proof}

\noindent Let $\mathcal{Q}\subseteq \omega\times 2^{\omega}$ be a
universal $\Pi^1_1$ set. In other words, $\mathcal{Q}$ is a
$\Pi^1_1$ set so that  every $\Pi^1_1$ set is some
$\mathcal{Q}_n=\{x\mid (n,x)\in \mathcal{Q}\}$. By Theorem 2.2.3 in
\cite{Kech73}, the real $x_0=\{n\mid \mu(\mathcal{Q}_n)=0\}$ is
$\Sigma^1_1$. Let
$$\mathfrak{c}=\{(n,\sigma)\mid n\in x_0 \wedge \exists x((n,x)\in
\mathcal{Q}\wedge \sigma\prec x)\}\subseteq \omega\times
2^{<\omega}.$$ Then $\mathfrak{c}$ can be viewed as a $\Sigma^1_2$
real. Since every $\Pi^1_1$ null closed set is
$\Pi^0_1(\mathfrak{c})$, every $\mathfrak{c}$-Kurtz random real is
$\Pi^1_1$-Kurtz random.

\begin{theorem}\label{theorem: base for pi11 random}
$\mathfrak{c}$ is a base for a cone of $\Pi^1_1$-Kurtz randoms.
\end{theorem}

\begin{proof}
For every real $y_0\geq_h \mathfrak{c}$, there is a real
$y_1\equiv_h y_0$ so that $y_1\geq_T \mathfrak{c}'$, the Turing jump
of $\mathfrak{c}$. By Lemma \ref{lemma Stephan lemma},  there is a
real $z \equiv_T y_1$ for which $z$ is $\mathfrak{c}$-Kurtz random
and so $\Pi^1_1$-Kurtz random.
\end{proof}

\noindent
Recall that every $\Sigma^1_2$ real is constructible (see e.g.\  the last chapter  of Moschovakis \cite{Mosch1980}). In the following we will determine the position of $\mathfrak{c}$ within the constructible hierarchy.
A real is called constructible if  it belongs to some level $L_\alpha$ of G\"odel's hierarchy of constructible sets
$$
L=\bigcup\{L_\beta:\beta\text{ is an ordinal}\}.
$$
More generally, for each real $x$ we have  the hierarchy
$$
L[x]=\bigcup\{L_\beta[x]:\beta\text{ is an ordinal}\}
$$
of sets constructible from  $x$.

  Let $$\delta^1_2=\sup \{ \alpha: \alpha \mbox{ is an ordinal isomorphic to a }\Delta^1_2 \mbox{
wellordering of } \omega\},$$ and
$$\delta=\min\{\alpha\mid L\setminus L_{\alpha}\mbox{ contains no
}\Pi^1_1 \mbox{ singleton}\}.$$

\begin{proposition}[Forklore]\label{proposition delta=delta12}
$\delta=\delta^1_2$.
\end{proposition}

\begin{proof}
If $\alpha<\delta$, then there is a $\Pi^1_1$ singleton $x \in
L_{\delta}\setminus L_{\alpha}$. Since $x\in L_{\omega_1^x}$ and
$\omega_1^x$ is a $\Pi^1_1(x)$ wellordering, it must be that
$\alpha<\omega_1^x<\delta^1_2$. So $\delta\leq \delta^1_2$.

If $\alpha<\delta^1_2$,  there is a $\Delta^1_2$ wellordering
relation $R\subseteq \omega\times \omega$ of order type $\alpha$. So
there are two recursive relations $S, T\subseteq
(\omega^{\omega})^2\times \omega^3$ so that
$$R(n,m)\Leftrightarrow \exists f \forall g \exists k S(f,g,n,m,k),
\mbox{ and}$$
$$\neg R(n,m)\Leftrightarrow \exists f \forall g \exists k T(f,g,n,m,k).$$
Define a $\Pi^1_1$ set $R_0=\{(f,n,m)\mid \forall g \exists k
S(f,g,n,m,k)\}$. By the Gandy-Spector Theorem \ref{theorem spector and
gandy}, there is an arithmetical relation $S'$ so that
$R_0=\{(f,n,m)\mid \exists g\leq_h f (S'(f,g,n,m))\}$. Recall that
every nonempty $\Pi^1_1$ set contains a $\Pi^1_1$-singleton (Kondo-Addison \cite{Sacks90}). Then
$$R(n,m)\Leftrightarrow \exists f\in L_{\delta} \exists g \in
L_{\omega_1^f}[f] (S'(f,g,n,m)).$$
In other words, $R$ is $\Sigma_1$-definable over $L_{\delta}$. By
the same method, the complement of $ R$ is $\Sigma_1$-definable over $L_{\delta}$
too. So $R$ is $\Delta_1$-definable over $L_{\delta}$. It is clear
that $L_{\delta}$ is admissible. So $R\in L_{\delta}$. Hence
$\alpha<\delta$. Thus $\delta^1_2=\delta$.
\end{proof}

\noindent

Note that if $x$ is a $\Delta^1_2$-real, then $\omega_1^x$ is
isomorphic to a $\Delta^1_2$ wellordering of $\omega$. So
$$\sup\{\omega_1^x\mid x \mbox{ is a
}\Pi^1_1\mbox{-singleton}\}\leq \delta^1_2.$$ \noindent Since $x\in
L_{\omega_1^x}$ for every $\Pi^1_1$-singleton $x$,
$$\sup\{\omega_1^x\mid x \mbox{ is a
}\Pi^1_1\mbox{-singleton}\}\geq \delta=\delta^1_2.$$\noindent Thus
$$\sup\{\omega_1^x\mid x \mbox{ is a }\Pi^1_1\mbox{-singleton}\}=
\delta=\delta^1_2.$$

\noindent Since every $\Pi^1_1$ singleton is recursive in
$\mathfrak{c}$, we have $\mathfrak{c}\not\in L_{\delta^1_2}$ and
$\omega_1^{\mathfrak{c}}\geq \delta^1_2$.

By the same argument as in Proposition \ref{proposition
delta=delta12}, the reals lying in $L_{\delta^1_2}$ are exactly the
$\Delta^1_2$ reals. So $\mathfrak{c}$ is not $\Delta^1_2$. Moreover,
since $\mathfrak{c}$ is $\Sigma^1_2$, it is $\Sigma_1$ definable
over $L_{\delta^1_2}$. Hence $\mathfrak{c} \in L_{\delta^1_2+1}$. In
other words, for any real $z$, if
$\omega_1^z>\omega_1^{\mathfrak{c}}$, then $\mathfrak{c}\in
L_{\omega_1^z}$ and so $\mathfrak{c}\leq_h z$. Then by \cite{Sa76},
$\mathfrak{c} \in L_{\omega_1^{\mathfrak{c}}}$. Thus
$\omega_1^{\mathfrak{c}}>\delta^1_2$. Since actually all
$\Sigma^1_2$ reals lie in $L_{\delta^1_2+1}$. This means that
$$\mathfrak{c}\mbox{ has the largest hyperdegree among all
}\Sigma^1_2 \mbox{ reals}.$$

\section{$\Delta^1_1$-traceability and dominability}

\noindent
We begin with  the characterization of $\QI$-randomness within
$\DI$-Kurtz randomness.

\begin{definition}
A real $x$ is hyp-dominated if for all functions $f:\omega
\to \omega$ with $f\leq_h x$, there is a hyperarithmetic function
$g$ so that $g(n)>f(n)$ for all $n$.
\end{definition}

\noindent
Recall that a real is $\Pi^1_1$-random if it does not belong to any
$\Pi^1_1$-null set. The following result is a higher  analog of  the
result that Kurtz randomness coincides with  weak 2-randomness  for
reals of  hyperimmune-free degree.

\begin{proposition}\label{proposition: delta11 same as pi11 if  dominability}
A real $x$ is $\Pi^1_1$-random if and only if $x$ is
hyp-dominated and $\Delta^1_1$-Kurtz random.\end{proposition}

\begin{proof}
Every $\Pi^1_1$-random real is $\Delta^1_1$-Kurtz random and also
hyp-dominated (see \cite{CNY07}). We prove the other direction.

Suppose $x$ is hyp-dominated and $\Delta^1_1$-Kurtz random. We show
that $x$ is $\Pi^1_1$-Martin-L\" of random. If not, then fix a
universal $\Pi^1_1$-Martin-L\" of test $\{U_n\}_{n\in \omega}$ (see
\cite{HjNi}). Then there is a recursive function $f:\omega\times
2^{<\omega}\to \omega$ so that for any pair $(n,\sigma)$, $\sigma\in
U_n$ if and only if $f(n,\sigma)\in \KO$. Since $x$ is
hyp-dominated, $\omega_1^x=\CK$ (see \cite{CNY07}). Then we define a
$\Pi^1_1(x)$ relation $R\subseteq \omega\times \omega$ so that
$R(n,m)$ if and only if there is a $\sigma$ so that $m\in \KO$,
$f(n,\sigma)\in \KO_m=\{i\in \KO\mid |i|<|m|\}$ and $\sigma\prec x$.
Then by the $\Pi^1_1$-uniformization relativized to $x$, there is a
partial function $p$ uniformizing $R$. Since $x\in \bigcap_n U_n$,
$p$ is a total function. Since $\omega_1^x=\CK$, there must be some
$m_0\in \KO$ so that $p(n)\in \KO_{m_0}$ for every $n$. Then define
a $\Delta^1_1$-Martin-L\" of test $\{\hat{U}_n\}_{n\in \omega}$ so
that $\sigma\in \hat{U}_n$ if and only if $f(n,\sigma)\in
\KO_{m_0}$. So $x \in \bigcap_n \hat{U_n}$. Let
$\hat{f}(n)=\min\{l\mid \exists \sigma \in 2^{l}(\sigma\in \hat{U}_n
\wedge x \in [\sigma])\}$ be a $\Delta^1_1(x)$ function. Then there
is a $\Delta^1_1$ function $f$ dominating $\hat{f}$. Define
$V_n=\{\sigma\mid \sigma \in 2^{\leq f(n)} \wedge \sigma \in
\hat{U}_n\}$ for every $n$. Then $P=\bigcap_n V_n$ is a $\Delta^1_1$
closed set and $x \in P$. So $x$ is not $\Delta^1_1$-Kurtz random, a
contradiction.

Since is   $\Pi^1_1$-Martin-L\" of random and $\omega_1^x=\CK$,
$x$ is already $\Pi^1_1$-random (see \cite{CNY07}).   \end{proof}

\noindent
Next we proceed to traceability.

\begin{definition}  \begin{itemize}  \item[(i)] Let $h: \omega
\rightarrow \omega$ be a nondecreasing unbounded
function that is hyperarithmetical.  A $\Delta^1_1$ trace with bound
$h$ is a uniformly   $\Delta^1_1$ sequence $(T_e)_{e\in
\omega}$ such that $|T_e| \le h(e)$ for each $e$.

\item[(ii)] $x \in 2^{\omega}$ is $\Delta^1_1$-traceable
\cite{CNY07}   if there is $h \in \Delta^1_1$ such that, for each
$f \le_h x$, there is a $\Delta^1_1$ trace with bound $h$ such
that, for each $e$, $f(e) \in T_e$.

\item[(iii)]  $x \in  2^{\omega}$ is $\Delta^1_1$-semi-traceable  if for each
$f \le_h x$, there is a $\Delta^1_1$ function $g$ so that, for
infinitely many $n$, $f(n)=g(n)$. We say that $g$ semi-traces $f$.

\item[(iv)] $x \in  2^{\omega}$ is $\Pi^1_1$-semi-traceable  if for each
$f \le_h x$, there is a partial $\Pi^1_1$ function $p$ so that, for
infinitely many $n$ we have $f(n)=p(n)$.
\end{itemize}
\end{definition}

\noindent
Note that, if  $(T_e)_{e \in \omega}$ is a uniformly $\Delta^1_1$
sequence of finite sets, then there is $g \in \Delta^1_1$ such that
for each $e$, $D_{g(e)} = T_e$ (where $D_n$  is the $n$th finite set
according to some recursive ordering). Thus
$$g(e)= \mu n \, \forall u \, [u \in D_n \leftrightarrow u \in T_e
].$$ In this formulation, the definition of $\Delta^1_1$
traceability is very close to that of
 recursive traceability.

Also  notice  that the choice of a  bound as a witness  for
traceability is immaterial:

\begin{proposition}[As in Terwijn and Zambella \cite{TZ}] \label{tracebound}
Let $A$ be a real that is $\DI$  traceable with bound $h$. Then $A$ is $\DI$
traceable with bound $h'$ for
any monotone and unbounded $\DI$  function $h'$.
\end{proposition}

\begin{lemma}\label{lemma semi pi11 trace=delta11 trace in dominated}
$x$ is $\Pi^1_1$-semi-traceable if and only if $x$ is
$\Delta^1_1$-semi-traceable.
\end{lemma}

\begin{proof}
It is not difficult to see that if $x$ is $\Pi^1_1$-semi-traceable,
then $\omega_1^x=\CK$. For otherwise, $x \geq_h \KO$. So it suffices
to show that $\KO$ is not $\Pi^1_1$-semi-traceable. Let
$\{\phi_i\}_{i\in\omega}$ be an effective enumeration of partial
recursive functions. Define a function $g\leq_T \KO'$ so that
$g(i)=\sum_{j\leq i}m^i_j+1$ where $m^i_j$ is the least number $k$
so that $p_j(i,k)\in \KO$; if there is no such $k$, then $m^i_j=0$.
Note that for any $\Pi^1_1$ partial function $p$, there must be some
partial recursive function $p_j$ so that for every pair $n, m$,
$p(n)=m$ if and only if $p_j(n,m)\in \KO$. Then by the definition of
$g$, for any $i>j$, $g(k)\neq p(i)$. So $g$ cannot be traced by $p$.

Suppose that $x$ is $\Pi^1_1$-semi-traceable, $\omega_1^x=\CK$,
and $f\leq_h x$. Fix a $\Pi^1_1$ partial function $p$ for $f$. Since
$p$ is a $\Pi^1_1$ function, there must be some recursive injection
$h$ so that $p(n)=m\Leftrightarrow h(n,m)\in \KO$.

Let $R(n,m)$ be a $\Pi^1_1(x)$ relation so that $R(n,m)$ iff there
exists $m>k\geq n$ for which $f(k)=p(k)$. Then some total  function $g$
uniformizes $R$ such that $g$   is
$\Pi^1_1(x)$, and so $\Delta^1_1(x)$. Thus,  for every $n$, there is some $m\in [g(n),
g(g(n)))$ so that $f(m)=p(m)$. Let $g'(0)=g(0)$, and
$g'(n+1)=g(g'(n))$ for all $n\in \omega$. Define a $\Pi^1_1(x)$
relation $S(n,m)$ so that $S(n,m)$ if and only if $m\in
[g'(n),g'(n+1))$ and $p(m)=f(m)$. Uniformizing $S$ we obtain a
$\Delta^1_1(x)$ function $g''$.

Define a $\Delta^1_1(x)$ set by $H=\{h(m,k)\mid \exists
n(g''(n)=m \wedge f(m)=k)\}$. Since $\omega_1^x=\CK$, $H\subseteq
\KO_n$ for some $n\in \KO$. Since $\KO_n$ is a $\Delta^1_1$ set, we
can define a $\Delta^1_1$ function $\hat{f}$ by: $\hat{f}(i)=j$
if $h(i,j)\in \KO_n$; $\hat{f}(i)=1$, otherwise. Then there are
infinitely many $i$ so that $f(i)=\hat{f}(i)$.
\end{proof}

\noindent

\noindent
Note that the $\Delta^1_1$-dominated reals form a measure 1 set
\cite{CNY07} but the set of $\Delta^1_1$-semi-traceable reals  is
null.
Chong, Nies and Yu \cite{CNY07} constructed a non-hyperarithmetic
$\Delta^1_1$-traceable real.

\begin{proposition}\label{proposition trace implies dom and
semitrace} Every $\Delta^1_1$-traceable real is
$\Delta^1_1$-dominated and $\Delta^1_1$-semi-traceable.
\end{proposition}

\begin{proof}
Obviously every $\Delta^1_1$-traceable real is
$\Delta^1_1$-dominated.

Suppose we are given a $\Delta^1_1$-traceable real $x$ and $\Delta^1_1(x)$ function
$f$. Let $g(n)=\langle f(2^{n}), f(2^{n}+2), \linebreak[3] \ldots,
f(2^{n+1}-1)\rangle$ for all $n \in \omega$. Then there is a
$\Delta^1_1$ trace $T$ for $g$ so that $|T_n|\leq n$ for all $n$.

Then for all $2^{n}+1\leq m\leq 2^{n+1}$, let $\hat{f}(m)=$ the
$(m-2^n)$-th entry  of the tuple of the $(m-2^n)$-th element of
$T_n$ if there exists such an $m$; otherwise, let $\hat{f}(m)=1$. It
is not difficult to see that  for every $n$  there is at least one
$m\in [2^n,2^{n+1})$ so that $f(m)=\hat{f}(m)$.
\end{proof}

\noindent
From the proof above, one can see the following corollary.

\begin{corollary}\label{corollary: characterizaton fo delta11 traceable}
A real $x$ is $\Delta^1_1$-traceable if and only if for every
$x$-hyperarithmetic $\hat{f}$, there is a hyperarithmetic function $f$
so that for every $n$, there is some $m\in [2^n, 2^{n+1})$ so that
$f(m)=\hat{f}(m)$.
\end{corollary}

\noindent
The following proposition will be used in Theorem \ref{theorem separating traceable and semi-traceable} to disprove the converse of Proposition \ref{proposition trace implies dom and semitrace}.

\begin{proposition}\label{proposition: Charaterization of semi+dominated}
For any real $x$, the following are equivalent.
 \begin{enumerate}
\item $x$ is $\Delta^1_1$-semi-traceable and $\Delta^1_1$-dominated.
\item For every function $g \leq_h x$, there exist an increasing
$\Delta^1_1$ function $f$ and a $\Delta^1_1$ function $F:\omega\to
[\omega]^{<\omega}$  with $|F(n)|\leq n$ so that for every $n$, there
exists some $m\in [f(n), f(n+1))$ with $g(m)\in F(m)$.
\end{enumerate}
\end{proposition}

\begin{proof}
(1)$\implies$ (2): Immediate because $1\le n$.

(2)$\implies$ (1). Suppose we are given a function $\hat{g}\leq_h x$. Without loss of
generality, $\hat{g}$ is nondecreasing. Let $f$ and $F$ be the corresponding $\Delta^1_1$ functions. Let $j(n)=\sum_{i\leq
f(n+1)}\sum_{k\in F(i)}k$ and note that $j$ is a $\Delta^1_1$ function
dominating $\hat{g}$.

To show that $x$ is $\Delta^1_1$-traceable, suppose we are given a function
$\hat{g}\leq_h x$. Let $h(n)=\langle g(2^{n}+1), \linebreak[3]
g(2^{n}+2),\ldots,g(2^{n+1}-1)\rangle$. Then by assumption there are corresponding
$\Delta^1_1$ functions $f_h$ and $F_h$. For every $n$ and $m \in
[2^n,2^{n+1})$, let $g(m)=$ the $(m-2^n)^{\text{th}}$ column of the $(m-2^n)^{\text{th}}$
element in $F_h(n)$ if such an $m$ exists; let $g(m)=1$ otherwise.
Then $g$ is a $\Delta^1_1$ function semi-tracing $\hat{g}$.
\end{proof}

\noindent
To separate $\Delta^1_1$-traceability from the conjunction of
$\Delta^1_1$-semi-traceability and  $\Delta^1_1$-do\-minability, we have
to modify Sacks' perfect set forcing.
\begin{definition}
\begin{enumerate}
\item  A $\Delta^1_1$ perfect tree $T \subseteq 2^{<\omega}$ is {\em
fat} at $n$ if for every $\sigma \in T$ with $|\sigma|\in [2^n,
2^{n+1})$, we have $\sigma^{\smallfrown}0 \in T$ and
$\sigma^{\smallfrown}1\in T$. Then we also say that $n$ is a
\emph{fat number} of $T$.
\item  A $\Delta^1_1$ perfect tree $T \subseteq 2^{<\omega}$ is {\em
clumpy} if there are infinitely many $n$ so that $T$ is fat at $n$.
\item  Let $\mathbb{F}=(\mathcal{F},\subseteq)$ be a partial
order of which the domain $\mathcal{F}$ is the collection of clumpy trees, ordered by inclusion.
\end{enumerate}\end{definition}

\noindent
Let $\varphi$ be a sentence of  $\mathfrak{L}(\CK,
\dot{x})$.  Then we can define the forcing relation, $T\Vdash
\varphi$, as done by Sacks in Section 4, IV \cite{Sacks90}.
\begin{enumerate}
\item $\varphi$ is ranked and $\forall x \in T(\mathfrak{A}(\CK,
x)\models \varphi)$, then $T \Vdash \varphi.$
\item If $\varphi(y)$ is unranked and $T \Vdash \varphi(\psi(n))$ for
some $\psi(n)$ of rank at most $\alpha$, then $T \Vdash \exists
y^{\alpha}\varphi(y^{\alpha})$.
\item If  $T \Vdash \exists y^{\alpha}\varphi(y^{\alpha})$, then $T
\Vdash \exists y \varphi(y)$.
\item If $\varphi(n)$ is unranked and $T \Vdash \varphi(m)$ for some
number $m$, then $T \Vdash \exists n\varphi(n)$.
\item If $\varphi$ and $\psi$ are unranked, $T \Vdash \varphi$ and $T
\Vdash \psi$, then $T \Vdash \varphi \wedge \psi$.
\item If $\varphi$ is unranked and $\forall P(P \subseteq T \implies P
\not\Vdash \varphi)$, then $T \Vdash \neg\varphi$.
\end{enumerate}
The following lemma can be deduced as done in \cite{Sacks90}.

\begin{lemma}\label{lemma comlexity of clumpy forcing}
 The relation $T \Vdash \varphi$, restricted to $\Sigma^1_1$
formulas $\varphi$, is $\Pi^1_1$.
\end{lemma}

\begin{lemma}\label{lemma fusion lemma}
\begin{enumerate}
\item Let $\{\varphi_i\}_{i\in \omega}$ be a hyperarithmetic sequence
of $\Sigma^1_1$ sentences. Suppose for every $i$ and $Q \subseteq T$,
there exists some $R \subseteq Q$ so that $R \Vdash \varphi_i$. Then
there exists some $Q \subseteq T$ so that for every $i$, $Q \Vdash \varphi_i$.
\item $\forall \varphi \forall T \exists Q \subseteq T (Q \Vdash
\varphi \vee Q\Vdash\neg\varphi)$.
\end{enumerate}
\end{lemma}

\begin{proof}
Using the notation $P \uh
n=\{\tau \in 2^{\leq n} \mid \tau \in P\}$, define $\mathcal{R}$ by
$$\mathcal{R}(R,i,\sigma, P)\Leftrightarrow (
\sigma \in R,\, P \subseteq R ,\,P \Vdash \varphi_i,\,
P\uh |\sigma|=\{\tau \mid \tau\prec\sigma\},$$
$$ \mbox{ and }\log|\sigma|-1\mbox{ is the }i^{\text{th}}\mbox{ fat number of } R).
$$

\noindent Note that $\mathcal{R}$ is a $\Pi^1_1$ relation. Then
$\mathcal{R}$ can be uniformized by a partial $\Pi^1_1$ function
$F:\mathcal{F}\times \omega \times 2^{<\omega}\to \mathcal{F}$.
Using $F$, a hyperarithmetic family $\{P_{\sigma}\mid \sigma\in
2^{<\omega}\}$ can be defined by recursion on $\sigma$.

$P_{\emptyset}=T.$

If $ \log|\sigma|-1 \mbox{ is not a fat number of } P_{\sigma}$, then
$P_{\sigma^{\smallfrown}0}, P_{\sigma^{\smallfrown}1}= P_{\sigma}.$

Otherwise: If $\sigma \not\in P_{\sigma}$, then
$P_{\sigma^{\smallfrown}0}= P_{\sigma^{\smallfrown}1}=\emptyset$.

  \ \ \ Otherwise:  $P_{\sigma^{\smallfrown}0}\cap
P_{\sigma^{\smallfrown}1}=\emptyset, P_{\sigma^{\smallfrown}0}\cup
P_{\sigma^{\smallfrown}1}\subseteq P_{\sigma},$

 \ \ \ $P_{\sigma^{\smallfrown}0}\uh |\sigma|, P_{\sigma^{\smallfrown}1}\uh
|\sigma|=\{\tau \mid \tau\prec\sigma\}$ and

 \ \ \ $P_{\sigma^{\smallfrown}0}, P_{\sigma^{\smallfrown}1}\Vdash
\wedge_{j\leq i}\varphi_{j}$ where

\ \ \ $i \mbox{ is the number so that } \log|\sigma|-1 \mbox{ is the }
       i \mbox{-th fat number of } T.$
\\
Let $Q=\bigcap_n \bigcup_{|\sigma|=n}P_{\sigma}$. Then $Q \in
\mathcal{F}$. It is routine to check that for every $i$, $Q \Vdash \varphi_i$.

\bigskip
\noindent
The proof of (2) is the same as the proof of Lemma 4.4 IV \cite{Sacks90}.
\end{proof}

\noindent
We say that a real $x$ is generic if it is the union of roots of trees in a generic filter; equivalently, for each $\Sigma^1_1$ sentence $\varphi$, there is
a condition $T$ such that $x \in T$ and either $T \Vdash \varphi$ or
$T \Vdash \neg\varphi$. One can check  (Lemma 4.8, IV \cite{Sacks90})  that for every
$\Sigma^1_1$-sentence $\varphi$,
$$\mathfrak{A}(\CK, x)\models \varphi \Leftrightarrow \exists P(x \in
P \wedge P\Vdash \varphi).$$

\begin{lemma}\label{lemma basics for generic}
If $x$ is a generic real, then
\begin{enumerate}
\item $\mathfrak{A}(\CK, x)$ satisfies $\Delta^1_1$-comprehension. So
$\omega_1^x=\CK$.
\item $x$ is $\Delta^1_1$-dominated  and $\Delta^1_1$-semi-traceable.
\item $x$ is not $\Delta^1_1$-traceable.
\end{enumerate}
\end{lemma}
\begin{proof}
(1). The proof of (1) is exactly same as the proof of Theorem 5.4 IV,
\cite{Sacks90}.

(2). By Proposition \ref{proposition: Charaterization of semi+dominated}, it suffices to show that  for every function $g\leq_h x$, there are an increasing $\Delta^1_1$ function $f$ and a $\Delta^1_1$ function $F:\omega\to \omega^{<\omega}$  with $|F(n)|\leq
n$ so that for every $n$, there exists some $m\in [f(n), f(n+1))$ so that $g(m)\in F(m)$. Since $g\leq_h x$ and $\omega_1^x=\CK$, there is a ranked formula $\varphi$ so that for every $n$, $g(n)=m$ if and only if $\mathfrak{A}(\CK, x)\models \varphi(n,m) $. So there is a condition $S \Vdash \forall n \exists! m \varphi(n,m)$. Fix a condition $T \subseteq S$. As in the proof of Lemma \ref{lemma fusion lemma}, we can build a hyperarithmetic sequence of conditions $\{P_{\sigma}\}_{\sigma\in 2^{<\omega}}$ so that
$$
P_{\sigma^{\smallfrown}i} \Vdash \varphi(|\sigma|, m_{\sigma^{\smallfrown}i}) \mbox{ for } i\leq 1
$$
if $\log|\sigma|-1$ is a fat number of $P_{\sigma}$ and $\sigma \in P_{\sigma}$.  Let $Q$ be as defined in the proof of Lemma \ref{lemma fusion lemma}. Let $f$ be the $\Delta^1_1$ function such that $f(0)=0$, and $f(n+1)$ is the least number $k>f(n)$ so that $m_{\sigma}$ is defined for some $\sigma$ with $f(n)<|\sigma|<k$. Let $F(n)=\{0\}\cup \{m_{\sigma}\mid |\sigma|=n\}$, and note that $F$ is a $\Delta^1_1$ function. Then
$$
Q \Vdash \forall n |F(n)|\leq n\wedge \forall n \exists m\in [f(n),f(n+1))\exists i \in
F(m)(\varphi(m, i)).
$$
So
$$
Q \Vdash \exists F \exists f (\forall n |F(n)|\leq n \wedge \forall n \exists m\in [f(n),f(n+1))\exists i \in F(m)(\varphi(m, i))).
$$
Since $T$ is an arbitrary condition stronger than $S$, this means
$$
S \Vdash \exists F \exists f (\forall n |F(n)|\leq n \wedge \forall n \exists m\in [f(n),f(n+1))\exists i \in F(m)(\varphi(m, i))).
$$
Since $x\in S$, $$\mathfrak{A}(\CK, x)\models\exists F \exists f (\forall n |F(n)|\leq n \wedge \forall n \exists m\in [f(n),f(n+1))\exists i \in F(m)(\varphi(m, i))).
$$
So $x$ is $\Delta^1_1$-dominated  and $\Delta^1_1$-semi-traceable.

(3). Suppose  $f:\omega\to \omega$ is a $\Delta^1_1$ function so that
for every $n$, there is a number $m \in [2^n, 2^{n+1})$ with
$f(m)=x(m)$. Then there is a ranked formula $\varphi$ so that $f(n)=m
\Leftrightarrow \mathfrak{A}(\CK, x)\models \varphi(n,m)$. Moreover,
$\mathfrak{A}(\CK, x)\models \forall n \exists m\in [2^n, 2^{n+1})
(\varphi(m,x(m))) $.  So there is a condition $T \Vdash \forall n
\exists m\in [2^n, 2^{n+1}) (\varphi(m,\dot{x}(m))) $ and $x \in T$.
Let $n$ be a number so that $T$ is fat at $n$ and $\sigma \in
2^{2^n-1}$ be a finite string in $T$. Let $\mu$ be a finite string so
that $\mu(m)=1-f(m+2^n-1)$. Define
$S=\{\sigma^{\smallfrown}\mu^{\smallfrown} \tau\mid
\sigma^{\smallfrown}\mu^{\smallfrown}\tau \in T\}\subseteq T$. Then $S
\Vdash \forall m\in [2^n, 2^{n+1})(\neg\varphi(m,x(m)))$. But $S$ is
stronger than $T$, a contradiction. By Corollary \ref{corollary:
characterizaton fo delta11 traceable}, $x$ is not $\Delta^1_1$-traceable.
\end{proof}

\noindent
We may now  separate $\Delta^1_1$-traceability from the conjunction of
 $\Delta^1_1$-semi-traceability and  $\Delta^1_1$-dominability.

\begin{theorem}\label{theorem separating traceable and semi-traceable}
There are $2^{\aleph_0}$ many  $\Delta^1_1$-dominated  and
$\Delta^1_1$-semi-traceable reals which are not $\Delta^1_1$-traceable.
\end{theorem}

\begin{proof}
This is
immediate from Lemma \ref{lemma basics for generic}. Note that there
are $2^{\aleph_0}$ many generic reals.
\end{proof}

\bigskip

\section{Lowness for higher Kurtz randomness}

Given a relativizable class of reals $\mathcal{C}$ (for instance, the class of random reals), we call a real $x$ \emph{low for}
$\mathcal{C}$ if $\mathcal{C}=\mathcal{C}^x$. We shall prove that lowness for $\Delta^1_1$-randomness is different
from lowness for $\Delta^1_1$-Kurtz randomness. A real $x$ is \emph{low for $\Delta^1_1$-Kurtz tests} if every $\Delta^1_1(x)$ open
set with measure 1 has a  $\Delta^1_1$ open subset of measure 1. Clearly, lowness for $\Delta^1_1$-Kurtz tests implies lowness for $\Delta^1_1$-Kurtz randomness.

\begin{theorem}\label{lemma: dom+semitrace implies low for delta kurtz}
If $x$ is $\Delta^1_1$-dominated and $\Delta^1_1$-semi-traceable, then $x$ is
low for $\Delta^1_1$-Kurtz tests.
\end{theorem}

\begin{proof}
Suppose $x$ is $\Delta^1_1$-dominated and $\Delta^1_1$-semi-traceable and $U$ is a
$\Delta^1_1(x)$ open set with measure 1. Then there is a real $y\leq_h
x$ so that $U$ is $\Sigma^0_1(y)$. Hence for some Turing reduction $\Phi$, if for all $z$ we write $U^z$ for the domain of $\Phi^z$, then we have $U=U^y$.

Define a $\Delta^1_1(x)$ function $\hat{f}$ by: $\hat{f}(n)$ is the shortest string
$\sigma \prec y $ so that $\mu(U^{\sigma}[\sigma])>1-2^{-n}$. By the
assumptions of the Theorem, there  are an increasing $\Delta^1_1$ function $g$ and a $\Delta^1_1$ function $f$ so that for every $n$, there is an
$m\in[g(n),g(n+1))$ so that $f(m)=\hat{f}(m)$.  Without loss of
generality, we can assume that $\mu(U^{f(m)}[m])>1-2^{-m}$ for every $m$.

Define a $\Delta^1_1$ open set $V$ so that $\sigma \in V$ if and only
if there exists some $n$ so that $[\sigma]\subseteq \bigcap_{g(n)\leq
m <g(n+1)}U^{f(m)}[m]$. By the property of $f$ and $g$, $V\subseteq
U^y=U$. But for every $n$, $$\mu(\bigcap_{g(n)\leq m
<g(n+1)}U^{f(m)}[m])>1-\sum_{g(n)\leq m <g(n+1)}2^{-m}\geq 1-2^{-g(n)+1}.$$
So $$\mu(V)\geq \lim_n \mu\left(\bigcap_{g(n)\leq m <g(n+1)}U^{f(m)}[m]\right)=1.$$
Hence $x$ is low for $\Delta^1_1$-Kurtz tests.
\end{proof}

\begin{corollary}
Lowness for $\Delta^1_1$-randomness differs  from lowness for
$\Delta^1_1$-Kurtz randomness.
\end{corollary}
\begin{proof}
By Theorem \ref{theorem separating traceable and semi-traceable}, there is a real $x$ that is $\Delta^1_1$-dominated and $\Delta^1_1$-semi-traceable but not $\Delta^1_1$-traceable. By Theorem \ref{lemma: dom+semitrace implies low for delta kurtz}, $x$ is low for $\Delta^1_1$-Kurtz randomness.
Chong, Nies and Yu \cite{CNY07} proved that lowness for
$\Delta^1_1$-randomness is the same as $\Delta^1_1$-traceability.  Thus $x$ is not low for $\Delta^1_1$-randomness.
\end{proof}

\begin{corollary}\label{corollary: nonbase for delta11 kurtz random}
There is a non-zero hyperdegree below $\KO$ which is not a base for  a
cone of $\Delta^1_1$-Kurtz randoms.
\end{corollary}

\begin{proof}
Clearly there is a real $x<_h \KO$ which is  $\Delta^1_1$-dominated
and $\Delta^1_1$-semi-traceable. Then the hyperdegree of $x$ is not a base for
a cone of $\Delta^1_1$-Kurtz randoms.
\end{proof}

\noindent
Actually the converse of Theorem \ref{lemma: dom+semitrace implies low
for delta kurtz} is also true.

\begin{lemma}\label{lemma: low for kutrz implies dominated}
If $x$ is low for  $\Delta^1_1$-Kurtz randomness, then $x$ is
$\Delta^1_1$-dominated.
\end{lemma}

\begin{proof}
Firstly we show that if $x$ is low for $\Delta^1_1$-Kurtz tests, then $x$
is  $\Delta^1_1$-dominated.

Suppose $f\leq_h x$ is an increasing function. Let $S_f=\{z\mid
\forall n(z(f(n))=0)\}$. Obviously $S_f$ is a $\Delta^1_1(x)$ closed
null set. So there is a  $\Delta^1_1$ closed null set $[T]\supseteq
S_f$ where $T\subseteq 2^{<\omega}$ is a $\Delta^1_1$ tree. Define
$$g(n)=\min\{m\mid \frac{|\{\sigma \in 2^m\mid \sigma \in
T\}|}{2^m}<2^{-n}\}+1.$$ Since $\mu([T])=0$, $g$ is a well defined
$\Delta^1_1$ function. We claim that $g$ dominates $f$.

For every $n$, $S_{f(n)}=\{\sigma \in 2^{f(n)}\mid \forall i\leq n(
\sigma(f(i))=0)\}$ has cardinality $2^{f(n)-n}$. But if $g(n)\leq
f(n)$, then since $S \subseteq [T]$, we have $$|S_{f(n)}|\leq
2^{f(n)-g(n)}\cdot |\{\sigma \in 2^{g(n)}\mid \sigma \in
T\}|<2^{f(n)-g(n)}\cdot 2^{g(n)-n}=2^{f(n)-n}.$$
This is a contradiction.
So $x$ is $\Delta^1_1$-dominated.

\bigskip
\noindent
Now suppose $x$ is not $\Delta^1_1$-dominated witnessed by some
$f\leq_h x$. Then $S_f$ is not contained in any $\Delta^1_1$ closed
null set. Actually, it is not difficult to see that for any $\sigma$
with $[\sigma]\cap S_f \neq \emptyset$, $[\sigma]\cap S_f$ is not
contained in any $\Delta^1_1$ closed null set (otherwise, as proved
above,  one can show that $f$ is dominated by some $\Delta^1_1$
function). Then, by an induction, we can construct a
$\Delta^1_1$-Kurtz random real $z\in S_f$ as follows:

Fix an enumeration $P_0, P_1,\ldots$ of the $\Delta^1_1$ closed null sets.

At stage $n+1$, we have constructed some $z \uh l_n$ so that $[z]\uh
l_n \cap S_f \neq \emptyset$. Then there is a $\tau \succ z \uh l_n$
so that $[\tau]\cap S_f \neq \emptyset$ but $[\tau]\cap S_f \cap P_n
=\emptyset$. Fix such a $\tau$, let $l_{n+1}=|\tau|$ and $z\uh l_{n+1}=\tau$.

Then $z\in S_f$ is $\Delta^1_1$-Kurtz random.

So $x$ is not low for $\Delta^1_1$-Kurtz randomness.
\end{proof}

\begin{lemma}\label{lemma: low for kurtz implies semitrace}
If $x$ is low for $\Delta^1_1$-Kurtz randomness, then $x$ is
$\Delta^1_1$-semi-traceable.
\end{lemma}

\begin{proof}
The proof is analogous to that of the main result in \cite{GrMi}.

Firstly we show that if $x$ is low for $\Delta^1_1$-Kurtz tests, then $x$
is  $\Delta^1_1$-semi-traceable.

Suppose that $x$ is low for  $\Delta^1_1$-Kurtz tests and $f\leq_h x$.
Partition $\omega$ into finite intervals $D_{m,k}$ for $0<k<m$ so that
$|D_{m,k}|=2^{m-k-1}$. Moreover, if $m<m'$, then $\max D_{m,k}<\min
D_{m',k'}$ for any $k<m$ and $k'<m'$. Let $n_m=\max \{i\mid i\in
D_{m,k}\wedge k<m\}$ for every $m\in \omega$.  Note that
$\{n_m\}_{m\in \omega}$ is a recursive increasing sequence.

For every function $h$, let $$P^h=\{x\in 2^{\omega}\mid  \forall m(
x(h \uh n_m)=0)\}$$ be a closed null set. Obviously $P^f$ is a
$\Delta^1_1(x)$ closed null set. Then there is a $\Delta^1_1$ closed
null set $Q \supseteq P^f$. We define a $\Delta^1_1$ function $g$ as follows.

\bigskip

\noindent
For each $k\in \omega$, let $d_k$ be the least number $d$ so that
$$|\{\sigma\in 2^d\mid \exists x \in Q (x\succ \sigma)\}|\leq 2^{d-k-1}.$$
Note that $\{d_k\}_{k \in \omega}$ is a
$\Delta^1_1$ sequence. Define $$Q_k=\{\sigma\mid \sigma \in 2^{d_k}
\wedge  \exists x \in Q (x\succ \sigma)\}.$$
Then $\{Q_k\}_{k\in \omega}$ is a $\Delta^1_1$ sequence of clopen sets
and $|Q_k|\leq 2^{d_k-k-1}$ for each $k<d_k$. Then Greenberg and
Miller \cite{GrMi} constructed a finite tree $S\subseteq
\omega^{<\omega}$ and a finite sequence $\{S_m\}_{k<m\leq l}$ for some
$l$ with the following properties:
\begin{enumerate}
\item $[S]=\{h\in \omega^{\omega}\mid P^h\subseteq [Q_k]\}$;
\item $S_m\subseteq S \cap \omega^{n_m}$;
\item $|S_m|\leq 2^{m-k-1}$;
\item every leaf of $S$ extends some string in $\bigcup_{k<m\leq l}S_m$.
\end{enumerate}
Moreover, both the finite tree $S$ and sequence $\{S_m\}_{k<m\leq l}$ can be obtained
uniformly from $Q_k$.

Now for each $m$ with $k<m\leq l$ and $\sigma \in S_m$,  we pick a
distinct $i \in D_{m,k}$ and define $g(i)=\sigma(i)$. For the other
undefined $i\in D_{m,k}$, let $g(i)=0$.

\bigskip

\noindent
So $g$ is a well-defined $\Delta^1_1$ function.

For each $k$, $P^f\subseteq Q\subseteq [Q_{k}]$. So $f \in [S]$. Hence
there must be some $i>n_k$ so that $f(i)=g(i)$.

Thus $x$ is $\Delta^1_1$-semi-traceable.

\bigskip

\noindent
Now suppose $x$ is not  $\Delta^1_1$-semi-traceable as witnessed by $f\leq_h x$. Then $P^f$ is not contained in any
$\Delta^1_1$ closed null set. It is shown in \cite{GrMi} that for any
$\sigma$, assuming that $[\sigma]\cap P^f\neq \emptyset$,
$[\sigma]\cap P^f$ is not contained in any $\Delta^1_1$ closed null
set. Then by an easy induction, one can construct a $\Delta^1_1$-Kurtz
random real in $P^f$.

So $x$ is not low for $\Delta^1_1$-Kurtz randomness.
\end{proof}

\noindent
So we have the following theorem.
\begin{theorem}
For any real $x\in 2^{\omega}$, the following are equivalent:
\begin{itemize}
\item[(1)] $x$  is low for $\Delta^1_1$-Kurtz tests;
\item[(2)] $x$ is low for $\Delta^1_1$-Kurtz randomness;
\item[(3)] $x$ is $\Delta^1_1$-dominated and $\Delta^1_1$-semi-traceable.
\end{itemize}
\end{theorem}

\noindent
It is unknown whether there exists a nonhyperarithmetic real which is
low for $\Pi^1_1$-Kurtz randomness.
However, we can prove the following containment.

\begin{proposition}\label{proposition low pi11 implies low delta11}
If $x$ is low for $\Pi^1_1$-Kurtz randomness, then $x$ is low for
$\Delta^1_1$-Kurtz randomness.
\end{proposition}

\begin{proof}
Assume that $x$ is low for $\Pi^1_1$-Kurtz randomness, $y$ is
$\Delta^1_1$-Kurtz random and there is a $\Delta^1_1(x)$ closed null
set $A$ with $y \in A$.  By Theorem \ref{theorem pi11 measure is
pi},  the set $$B=\bigcup\{C\mid C \mbox{ is a }\Delta^1_1 \mbox{
closed null set}\}$$ is a $\Pi^1_1$ null set. So $A-B$ is a
$\Sigma^1_1(x)$set. Since $y$ is $\Delta^1_1$-Kurtz random,
$y\not\in B$. Hence $y\in A-B$ and so $A-B$ is a $\Sigma^1_1(x)$
nonempty set. Thus there must be some real $z\in A-B$ with
$\omega_1^z=\omega_1^x=\CK$. Since $z \not\in B$, $z$ is
$\Delta^1_1$-Kurtz random.  So by Proposition \ref{proposition:
delta11 vs pi11 at ck}, $z$ is $\Pi^1_1$-Kurtz random. This
contradicts the fact that $x$ is low for $\Pi^1_1$-Kurtz randomness.
\end{proof}

\bibliographystyle{plain}

\end{document}